\documentclass[12pt]{article}
\usepackage{latexsym}
\usepackage{amsmath, amsthm, amsfonts}
\usepackage{graphics}
\usepackage{graphicx}
\usepackage{color}
\usepackage[table,xcdraw]{xcolor}
\usepackage{float}
\usepackage{array}
\usepackage{hyperref}

\newcommand{\abs}[1]{\left\lvert#1\right\rvert}
\newtheorem{theorem}{Theorem}[section]

\newtheorem{openproblem}[theorem]{Open Problem}

\def\smallskip{\addvspace{\smallskipamount}}
\def\medskip{\addvspace{\medskipamount}}
\def\bigskip{\addvspace{\bigskipamount}}
\def\makefootline{\baselineskip=24pt \line{\the\footline}}
\def\pagecontents{\ifvoid\topins\else\unvbox\topins\fi
   \dimen@=\dp255 \unvbox255
   \ifvoid\footins\else
      \vskip\skip\footins \footnoterule \unvbox\footins\fi
     \ifr@ggedbottom \kern-\dimen@ \vfil \fi}
\def\footnoterule{\kern-3pt\hrule width 2truein \kern 2.6pt}

\begin{document}

\title{{Complex Dynamics of $\displaystyle{z_{n+1}=\frac{\alpha + \beta z_{n}+ \gamma z_{n-1}}{A + B z_n + C z_{n-1}}}$}}

\author{Sk. Sarif Hassan\\
  \small {International Centre for Theoretical Sciences}\\
  \small {Tata Institute of Fundamental Research}\\
  \small {Bangalore $560012$, India}\\
  \small Email: {\texttt{sarif.hassan@icts.res.in}}\\
}

\maketitle
\begin{abstract}
\noindent The dynamics of the second order rational difference equation $\displaystyle{
z_{n+1}=\frac{\alpha + \beta z_{n}+ \gamma z_{n-1}}{A + B z_n + C z_{n-1}}}$ with complex parameters and arbitrary complex initial conditions is investigated. In the complex set up, the local asymptotic stability and boundedness are studied vividly for this difference equation. Several interesting characteristics of the solutions of this equation, using computations, which does not arise when we consider the same equation with positive real parameters and initial conditions are shown. The chaotic solutions of the difference equation is absolutely new feature in the complex set up which is also shown in this article.  Some of the interesting observations led us to pose some open interesting problems regarding chaotic and higher order periodic solutions and global asymptotic convergence of this equation.
\end{abstract}

\vfill
\begin{flushleft}\footnotesize
{Keywords: Rational difference equation, Local asymptotic stability, Chaotic trajectory and Periodicity. }
\end{flushleft}

\section{Introduction and Preliminaries}
A rational difference equation is a nonlinear difference equation of the form
$$z_{n+1} = \frac{\alpha+\sum_{i=0}^k \beta_iz_{n-i}}{A+\sum_{i=0}^k B_iz_{n-i}}$$
where the initial conditions $z_{0}, z_{-1}, \dots, z_{-k}$ are such that the denominator never vanishes for any $n$.\\
\noindent
Consider the equation

\begin{equation}
z_{n+1}=\frac{\alpha + \beta z_{n}+ \gamma z_{n-1}}{A + B z_n + C z_{n-1}},  n=0,1,\ldots
\label{equation:total-equationA}
\end{equation}%
where all the parameters and the initial conditions $z_{-1}$ and $z_{0}$ are arbitrary complex number.

\addvspace{\bigskipamount}
\noindent
This second order rational difference equation Eq.(\ref{equation:total-equationA}) is studied when the parameters are real numbers and initial conditions are non-negative real numbers in \cite{S-R-E}. In this present article it is an attempt to understand the same in the complex plane.\\

\noindent
Here, a very brief review of some well known results which will be useful in order to apprehend the behavior of solutions of the difference equation (\ref{equation:total-equationA}).

\bigskip
\noindent
Let $f:\mathbb{D}^{2} \rightarrow \mathbb{D}$ where $\mathbb{D}\subseteq \mathbf{C}$ be a continuously differentiable function. Then for any pair of initial conditions $z_0, z_{-1} \in \mathbb{D}$, the difference equation
\begin{equation}
\label{equation:introduction}
z_{n+1} = f(z_{n}, z_{n-1}) \hspace{.25in}
 \hspace{.25in}
\end{equation}
with initial conditions $z_{-1}, z_{0} \in \mathbb{D}.$\\ \\
\noindent
Then for any \emph{initial value}, the difference equation (1) will have a unique solution $\{z_n\}_n$.\\ \\
\noindent
A point $\overline{z}$ $\in$ $\mathbb{D}$ is called {\it \textbf{equilibrium point}} of Eq.(\ref{equation:introduction}) if
$$f(\overline{z}, \overline{z}) = \overline{z}.$$
\noindent
The \emph{linearized equation} of Eq.(\ref{equation:introduction}) about the equilibrium $\bar{z}$ is the linear difference equation

\begin{equation}
\label{equation:linearized-equation}
\displaystyle{
z_{n+1} = a_{0} z_{n} + a_{1}z_{n-1} \hspace{.25in} , \hspace{.25in} n=0,1,\ldots
}
\end{equation}
where for $i = 0$ and $1$.
$$a_{i} = \frac{\partial f}{\partial u_{i}}(\overline{z}, \overline{z}).$$
The \emph{characteristic equation} of Eq.(2) is the equation

\begin{equation}
\label{equation:characteristic-roots}
\lambda^{2} - a_{0}\lambda -a_{1} = 0.
\end{equation}

\bigskip
\noindent
The following are the briefings of the linearized stability criterions which are useful in determining the local stability character of the equilibrium
$\overline{z}$ of Eq.(\ref{equation:introduction}), \cite{S-H}.

\bigskip
\noindent
Let $\overline{z}$ be an equilibrium of the difference equation $z_{n+1}=f(z_n,z_{n-1})$.
\begin{itemize}
\item The equilibrium $\bar{z}$ of Eq. (2) is called \textbf{locally stable} if for every $\epsilon >0$, there exists a $\delta>0$ such that for every $z_0$ and $z_{-1}$ $\in$ $\mathbb{C}$ with $|z_0-\bar{z}|+|z_{-1}-\bar{z}|<\delta$ we have $|z_n-\bar{z}|<\epsilon$ for all $n>-1$.
\item The equilibrium $\bar{z}$ of Eq. (2) is called \textbf{locally stable} if it is locally stable and if there exist a $\gamma>0$ such that for every $z_0$ and $z_{-1}$ $\in$ $\mathbb{C}$ with $|z_0-\bar{z}|+|z_{-1}-\bar{z}|<\gamma$ we have $\lim_{n \to\infty}z_n=\bar{z}$.
\item The equilibrium $\bar{z}$ of Eq. (2) is called \textbf{global attractor} if for every  $z_0$ and $z_{-1}$ $\in$ $\mathbb{C}$, we have $\lim_{n \to\infty}z_n=\bar{z}$.
\item The equilibrium of equation Eq. (2) is called \textbf{globally asymptotically stable/fit} is stable and is a global attractor.
\item The equilibrium $\bar{z}$ of Eq. (2) is called \textbf{unstable} if it is not stable.
\item The equilibrium $\bar{z}$ of Eq. (2) is called \textbf{source or repeller} if there exists $r>0$ such that for every $z_0$ and $z_{-1}$ $\in$ $\mathbb{C}$ with $|z_0-\bar{z}|+|z_{-1}-\bar{z}|<r$ we have $|z_n-\bar{z}|\geq r$. Clearly a source is an unstable equilibrium.
\end{itemize}

\noindent
\textbf{Result 1.1: (Clark's Theorem)} The sufficient condition for the asymptotic stability of the difference equation (1) is $$\abs{a_0}+\abs{a_1}<1$$

\section{Difference Equation and Its Transformed Forms}
The following difference equation is considered to be studied here.

\begin{equation}
z_{n+1}=\frac{\alpha + \beta z_{n}+ \gamma z_{n-1}}{A + B z_n + C z_{n-1}},  n=0,1,\ldots
\end{equation}%
where all the parameters are complex number and the initial conditions $z_{-1}$ and $z_{0}$ are arbitrary complex numbers. \\

\noindent
We will consider three different cases of the Eq.(\ref{equation:introduction}) which are as follows:\\

\subsection{The case $\beta=\gamma=0: z_{n+1}=\frac{\alpha}{A+B z_n+ C z_{n-1}}$}

By the change of variables, $z_n=\frac{\alpha}{A} w_n$, the difference equation $z_{n+1}=\frac{\alpha}{A+B z_n+ C z_{n-1}}$ reduced to the difference equation

\begin{equation}
w_{n+1}=\frac{1}{1 + p w_n + q w_{n-1}},  n=0,1,\ldots
\end{equation}%
where $p=\frac{\alpha B}{A^2}$ and $q=\frac{\alpha C}{A^2}$.

\subsection{The case $\alpha=\gamma=0: z_{n+1}=\frac{\beta z_n}{A+B z_n+ C z_{n-1}}$}

By the change of variables, $z_n=\frac{A}{C} w_n$, the difference equation $z_{n+1}=\frac{\beta z_n}{A+B z_n+ C z_{n-1}}$ reduced to the difference equation

\begin{equation}
w_{n+1}=\frac{w_n}{1 + p w_n + q w_{n-1}},  n=0,1,\ldots
\end{equation}%
where $p=\frac{\beta}{A}$ and $q=\frac{B}{C}$.

\subsection{The case $\alpha=\beta=0: z_{n+1}=\frac{\gamma z_{n-1}}{A+B z_n+ C z_{n-1}}$}

By the change of variables, $z_n=\frac{\gamma}{C} w_n$, the difference equation $z_{n+1}=\frac{\gamma z_{n-1}}{A+B z_n+ C z_{n-1}}$ reduced to the difference equation

\begin{equation}
w_{n+1}=\frac{w_{n-1}}{p + q w_n + w_{n-1}},  n=0,1,\ldots
\end{equation}%
where $p=\frac{A}{\gamma}$ and $q=\frac{B}{C}$.\\

\noindent
Without any loss of generality, we shall now onward focus only on the three difference equations (6), (7) and (8).

\section{Local Asymptotic Stability of the Equilibriums}
In this section we establish the local stability character of the equilibria of Eq.(\ref{equation:total-equationA}) in three difference cases as stated in the section 2.

\subsection{Local Asymptotic Stability of $w_{n+1}=\frac{1}{1 + p w_n + q w_{n-1}}$}
The equilibrium points of Eq.(6) are the solutions of the quadratic equation
\[
\bar{w}=\frac{1}{1+p\bar{w}+q\bar{w}}
\]
Eq.(6) has the two equilibria points
$ \bar{w}_{1,2} =\frac{-1-\sqrt{1+4 p+4 q}}{2 (p+q)}$ and $\frac{-1+\sqrt{1+4 p+4 q}}{2 (p+q)}$ respectively.
\noindent
The linearized equation of the rational difference equation(6) with respect to the equilibrium point $ \bar{w}_{1} = \frac{-1-\sqrt{1+4 p+4 q}}{2 (p+q)}$ is

\begin{equation}
\label{equation:linearized-equation}
\displaystyle{
z_{n+1} + \frac{4 p}{\left(-1+\sqrt{1+4 p+4 q}\right)^2} w_{n} + \frac{4 q}{\left(-1+\sqrt{1+4 p+4 q}\right)^2} w_{n-1}=0,  n=0,1,\ldots
}
\end{equation}

\noindent
with associated characteristic equation
\begin{equation}
\lambda^{2} + \frac{4 p}{\left(-1+\sqrt{1+4 p+4 q}\right)^2} \lambda + \frac{4 q}{\left(-1+\sqrt{1+4 p+4 q}\right)^2} = 0.
\end{equation}

\addvspace{\bigskipamount}

\noindent
The following result gives the local asymptotic stability of the equilibrium $\bar{w}_{1}$ of the Eq. (6).

\begin{theorem}
The equilibriums $\bar{w}_{1}=\frac{-1-\sqrt{1+4 p+4 q}}{2 (p+q)}$ of Eq.(6) is \\ \\ locally asymptotically stable if $$\abs{\frac{4 p}{\left(-1+\sqrt{1+4 p+4 q}\right)^2}}+\abs{{\frac{4 q}{\left(-1+\sqrt{1+4 p+4 q}\right)^2}}}<1$$
\end{theorem}

\addvspace{\bigskipamount}

\begin{proof}
The zeros of the characteristic equation (10) has two zeros which are $-\frac{2 q}{p+\sqrt{p^2-4 p q+2 q \left(-1-2 q+\sqrt{1+4 p+4 q}\right)}}$ and $\frac{2 q}{-p+\sqrt{p^2-4 p q+2 q \left(-1-2 q+\sqrt{1+4 p+4 q}\right)}}$. Therefore by \emph{Clark's theorem}, the equilibrium $\bar{w}_{1}=\frac{-1-\sqrt{1+4 p+4 q}}{2 (p+q)}$  is \emph{locally asymptotically stable} if the sum of the modulus of two coefficients is less than $1$. Therefore the condition of the polynomial (10) reduces to $\abs{\frac{4 p}{\left(-1+\sqrt{1+4 p+4 q}\right)^2}}+\abs{{\frac{4 q}{\left(-1+\sqrt{1+4 p+4 q}\right)^2}}}<1$.

\end{proof}

\noindent
The linearized equation of the rational difference equation (6) with respect to the equilibrium point $ \bar{w}_{2} = \frac{-1+\sqrt{1+4 p+4 q}}{2 (p+q)}$ is

\begin{equation}
\label{equation:linearized-equation}
\displaystyle{
w_{n+1} + \frac{4 p}{\left(1+\sqrt{1+4 p+4 q}\right)^2} w_{n} + \frac{4 q}{\left(1+\sqrt{1+4 p+4 q}\right)^2} w_{n-1}=0,  n=0,1,\ldots
}
\end{equation}

\noindent
with associated characteristic equation
\begin{equation}
\lambda^{2} + \frac{4 p}{\left(1+\sqrt{1+4 p+4 q}\right)^2} \lambda + -\frac{4 q}{\left(1+\sqrt{1+4 p+4 q}\right)^2} = 0.
\end{equation}

\addvspace{\bigskipamount}

\begin{theorem}
The equilibriums $\bar{w}_{2}=\frac{-1+\sqrt{1+4 p+4 q}}{2 (p+q)}$ of Eq.(6) is \\ \\ locally asymptotically stable if $$\abs{\frac{4 p}{\left(1+\sqrt{1+4 p+4 q}\right)^2}}+\abs{\frac{4 q}{\left(1+\sqrt{1+4 p+4 q}\right)^2}}<1$$
\end{theorem}

\begin{proof}
Proof the theorem follows from \emph{Clark's theorem} of local asymptotic stability of the equilibriums. The condition for the local asymptotic stability reduces to $\abs{\frac{4 p}{\left(1+\sqrt{1+4 p+4 q}\right)^2}}+\abs{\frac{4 q}{\left(1+\sqrt{1+4 p+4 q}\right)^2}}<1$.

\end{proof}

\noindent
Here is an example case for the local asymptotic stability of the equilibriums.\\

\noindent
For $p=\frac{1}{2}$ and $q=\frac{i}{2}$ the equilibriums are $-1.6838+1.13355 i$ and $0.683802 -0.133552 i$. For the equilibrium $1.6838 + 1.13355 i$, the coefficients of the characteristic polynomial (10) are $0.775125+1.90868 i$ and $-1.90868 - 0.775125 i$ with same modulus $2.06006$. Therefore the condition as stated in the \emph{Theorem 3.1} does not hold. Therefore the  equilibrium $1.6838 + 1.13355 i$ is \emph{unstable}.\\

\noindent
For the equilibrium $0.683802 -0.133552 i$, the coefficients of the characteristic polynomial (12) are $-0.224875+0.091323 i$ and $-0.091323 - 0.224875 i$ with same modulus $0.242711$. Therefore the condition as stated in the \emph{Theorem 3.2} is hold good. Therefore the  equilibrium $0.683802 -0.133552 i$ is \emph{locally asymptotically stable}.\\

\noindent
It is seen that in the case of real positive parameters and initials values, the positive equilibrium of the difference equation (6) is globally asymptotically stable \cite{S-R-E}. But the result is not holding well in the complex set up.

\subsection{Local Asymptotic Stability of $w_{n+1}=\frac{w_n}{1 + p w_n + q w_{n-1}}$}
The equilibrium points of Eq.(7) are the solutions of the quadratic equation
\[
\bar{w}=\frac{\bar{w}}{1+p\bar{w}+q\bar{w}}
\]
The Eq.(7) has only the zero equilibrium.
\noindent
The linearized equation of the rational difference equation(7) with respect to the zero equilibrium is

\begin{equation}
\label{equation:linearized-equation}
\displaystyle{
w_{n+1} =  w_{n},  n=0,1,\ldots
}
\end{equation}

\noindent
with associated characteristic equation
\begin{equation}
\lambda^{2} - \lambda = 0.
\end{equation}

\addvspace{\bigskipamount}

\noindent
The following result gives the local asymptotic stability of the zero equilibrium of the Eq. (7).

\begin{theorem}
The zero equilibriums of the Eq. (7) is non-hyperbolic.
\end{theorem}

\addvspace{\bigskipamount}

\begin{proof}
The zeros of the characteristic equation (14) has two zeros which are $0$ and $1$. Therefore by definition, the zero equilibrium is non-hyperbolic as the modulus of one zero is $1$.
\end{proof}

\noindent
It is nice to note that in the case of real positive parameters and initials values, the zero equilibrium of the difference equation (7) is globally asymptotically stable for the parameter $p \geq 1$ \cite{S-R-E}. But in the case of complex, the zero equilibrium is non-hyperbolic as we have seen the previous theorem.

\subsection{Local Asymptotic Stability of $w_{n+1}=\frac{w_{n-1}}{p + q w_n + w_{n-1}}$}
The equilibrium points of Eq.(8) are the solutions of the quadratic equation
\[
\bar{w}=\frac{\bar{w}}{p+q\bar{w}+\bar{w}}
\]
The Eq.(8) has two equilibriums which are $0$ and $\frac{1-p}{1+q}$.
\noindent
The linearized equation of the rational difference equation(8) with respect to the zero equilibrium is

\begin{equation}
\label{equation:linearized-equation}
\displaystyle{
w_{n+1} =  \frac{1}{p}w_{n},  n=0,1,\ldots
}
\end{equation}

\noindent
with associated characteristic equation
\begin{equation}
\lambda^{2} - \frac{\lambda}{p} = 0.
\end{equation}

\addvspace{\bigskipamount}

\noindent
The following result gives the local asymptotic stability of the zero equilibrium of the Eq. (8).

\begin{theorem}
The zero equilibriums of the Eq. (8) is locally asymptotically stable if $\abs{p} \geq 1$ and repeller if $\abs{p}<1$.
\end{theorem}

\addvspace{\bigskipamount}

\begin{proof}
The zeros of the characteristic equation (15) has two zeros which are $0$ and $\frac{1}{p}$. Therefore by definition, the zero equilibrium is locally asymptotically stable if $\abs{\frac{1}{p}} <1$ and unstable (repeller) if $\abs{\frac{1}{p}} \geq  1$. Hence the required is followed.
\end{proof}

\bigskip

\bigskip

\noindent
The linearized equation of the rational difference equation(8) with respect to the equilibrium $\frac{1-p}{1+q}$ is

\begin{equation}
\label{equation:linearized-equation}
\displaystyle{
w_{n+1}=\frac{1+pq}{1+q} w_n + \frac{p-1}{1+q} w_{n-1}=0,  n=0,1,\ldots
}
\end{equation}

\noindent
with associated characteristic equation
\begin{equation}
\lambda^{2} - \frac{1+pq}{1+q} \lambda - \frac{p-1}{1+q} = 0.
\end{equation}

\addvspace{\bigskipamount}

\noindent
The following result gives the local asymptotic stability of the equilibrium $\frac{1-p}{1+q}$ of the Eq. (8).

\begin{theorem}
The zero equilibriums of the Eq. (8) is locally asymptotically stable if $$\abs{1+pq}+\abs{p-1}<\abs{1+q}$$
\end{theorem}

\addvspace{\bigskipamount}

\begin{proof}
The equilibrium $\frac{1-p}{1+q}$  of the characteristic equation (18) would be \emph{locally asymptotically stable} if the sum of the modulus of the coefficients of the characteristic equation (18) is less than 1. That is by \emph{Clark's theorem}, $\abs{\frac{1+pq}{1+q}}+\abs{\frac{p-1}{1+q}}<1$, that is $$\abs{1+pq}+\abs{p-1}<\abs{1+q}$$ 
\end{proof}

\noindent
Here is an example case for the local asymptotic stability of the equilibriums.\\

\noindent
For $p=1+\frac{i}{2}$ $(\abs{p}=1.1180>1$) and $q=\frac{1}{10} + i$ ($\abs{q}=1.005>1$) the equilibriums are $0$ and $-0.226244-0.248869 i$. For the equilibrium $-0.226244-0.248869 i$, the coefficients of the characteristic polynomial (18) are $0.533622 +0.545503 i$ and $-0.203409 - 0.0330197i$ with modulus $0.763103$ and $0.206072$ respectively. Therefore the condition as stated in the \emph{Theorem 3.5} hold good. Therefore the  equilibrium $1.6838 + 1.13355 i$ is \emph{locally asymptotically stable}.\\

\noindent
In the case of real positive parameters and intimal values of the difference equation (8), the positive equilibrium is locally asymptotically stable if $p<1$ and $q<1$ but in the complex set, it is encountered through the example above is that the equilibrium $\frac{1-p}{1+q}$ is locally asymptotically stable even though $\abs{p}>1$ and $\abs{q}>1$ \cite{S-R-E}.

\section{Boundedness}
In this section we would like to explore the boundedness of the solutions of the three difference equations (6), (7) and (8). \\

Now we would like to try to find open ball $B(0, \epsilon) \in \mathbf{C}$ such that if $w_n \in B(0, \epsilon)$ and $w_{n-1} \in B(0, \epsilon)$ then $w_{n+1} \in B(0, \epsilon)$ for all $n \geq 0$. In other words, if the initial values $w_0$ and $w_{-1}$ belong to $B(0, \epsilon)$ then the solution generated by the difference equations would essentially be within the open ball $B(0, \epsilon)$.
\begin{theorem}
\label{Result:positive-local-stability2} For the difference equation (6), for every $\epsilon >0$, if $w_n$ and $w_{n-1}$ $\in B(0, \epsilon)$ then $w_{n+1} \in B(0, \epsilon)$ provided $$\abs{p} \geq 1+ \abs{q}$$
\end{theorem}

\noindent

\begin{proof}
Let $\{w_{n}\}$ be a solution of the equation Eq.(6). Let $\epsilon>0$ be any arbitrary real number. Consider $w_n, w_{n-1} \in B(0, \epsilon)$. We need to find out an $\epsilon$ such that $w_{n+1}\in B(0, \epsilon)$ for all $n$. It is follows from the Eq.(6) that for any $\epsilon >0$, using Triangle inequality for $$\abs{w_{n+1}}=\abs{\frac{1}{1+p w_n+q w_{n-1}}} \leq \abs{\frac{1}{p w_n+q w_{n-1}}} \leq \frac{1}{(\abs{p}-\abs{q})\epsilon}$$ In order to ensure that $\abs{w_{n+1}}<\epsilon '$, (Assuming $\epsilon '=\frac{1}{\epsilon}$) it is needed to be $$\frac{1}{(\abs{p}-\abs{q})}<1$$ That is $\abs{p} \geq 1+ \abs{q}$.
\noindent
Therefore the required is followed.

\end{proof}

\begin{theorem}
\label{Result:positive-local-stability2} For the difference equation (7), for every $\epsilon >0$, if $w_n$ and $w_{n-1}$ $\in B(0, \epsilon)$ then $w_{n+1} \in B(0, \epsilon)$ provided $$\abs{p} \geq \abs{q}+ \frac{1}{\epsilon}$$
\end{theorem}

\noindent

\begin{proof}
Let $\{w_{n}\}$ be a solution of the equation Eq.(7). Let $\epsilon>0$ be any arbitrary real number. Consider $w_n, w_{n-1} \in B(0, \epsilon)$. We need to find out an $\epsilon$ such that $w_{n+1}\in B(0, \epsilon)$ for all $n$. It is follows from the Eq.(7), that for any $\epsilon >0$, using Triangle inequality for $$\abs{w_{n+1}}=\abs{\frac{w_n}{1+p w_n+q w_{n-1}}} \leq \frac{\epsilon}{(\abs{p}-\abs{q})\epsilon}$$ Therefore, $\abs{w_{n+1}}<\frac{1}{\abs{p}-\abs{q}}$. We need $\frac{1}{\abs{p}-\abs{q}}$ to be less than $\epsilon$. Therefore, $$\abs{p} \geq \abs{q}+ \frac{1}{\epsilon}$$ is followed.

\end{proof}

\begin{theorem}
\label{Result:positive-local-stability2} For the difference equation (8), for every $\epsilon >0$, if $w_n$ and $w_{n-1}$ $\in B(0, \epsilon)$ then $w_{n+1} \in B(0, \epsilon)$ and $\abs{p}<1$ provided $$\epsilon < \frac{\abs{p}-1}{\abs{q}+1}$$
\end{theorem}

\noindent

\begin{proof}
Let $\{w_{n}\}$ be a solution of the equation Eq.(8). Let $\epsilon>0$ be any arbitrary real number and $\abs{p}>1$. Consider $w_n, w_{n-1} \in B(0, \epsilon)$. We need to find out an $\epsilon$ such that $w_{n+1}\in B(0, \epsilon)$ for all $n$. It is follows from the Eq.(8), that for any $\epsilon >0$, using Triangle inequality for $$\abs{w_{n+1}}=\abs{\frac{w_{n-1}}{p+q w_n+w_{n-1}}} \leq \frac{\epsilon}{\abs{p}-\abs{q}\epsilon-\epsilon}$$ Therefore, $$\abs{w_{n+1}}\leq \frac{\epsilon}{\abs{p}-\abs{q}\epsilon-\epsilon}$$ In order to ensure that $\abs{w_{n+1}}<\epsilon$, it is needed to be $$\frac{1}{\abs{p}-\abs{q}\epsilon-\epsilon}<1$$ That is $$\epsilon < \frac{\abs{p}-1}{\abs{q}+1}$$
\noindent
Therefore the required is followed.

\end{proof}

\section{Periodic of Solutions}

\label{section:periodicity}

A solution $\{w_n\}_n$ of a difference equation is said to be \emph{globally periodic} of period $t$ if $w_{n+t}=w_n$ for any given initial conditions. solution $\{w_n\}_n$
is said to be \emph{periodic with prime period} $p$ if p is the smallest positive integer having this property.\\

\noindent
We shall first look for the prime period two solutions of the three difference equations (6), (7) and (8) and their corresponding local stability analysis.

\subsection{Prime Period Two Solutions of Eq. (6)}

Let $\ldots, \phi ,~\psi , ~\phi , ~\psi ,\ldots$, $\phi  \neq \psi $ be a prime period two solution of the difference equation $w_{n+1}=\frac{1}{1 + p w_n + q w_{n-1}}$. Then $\phi=\frac{1}{1+p \psi+q \phi}$ and $\psi=\frac{1}{1+p \phi +q \psi}$. This two equations lead to the set of solutions (prime period two) except the equilibriums as $\left\{\phi \to \frac{0.5 \left(p-q+\sqrt{p^2+p (-2 -4 q) q+q^2+4 q^3}\right)}{q (-p+q)}, \psi \to \frac{0.5 p-0.5 q-0.5 \sqrt{p^2+p (-2-4 q) q+q^2+4 q^3}}{q (-p+q)}\right\}$ and $\left\{\phi \to \frac{q}{-0.5 p+0.5 q-0.5 \sqrt{p^2+p (-2-4q) q+q^2+4q^3}}, \psi \to  \frac{-0.5-\frac{0.5 \sqrt{p^2+p (-2-4q) q+q^2+4q^3}}{p- q}}{q}\right\}$.\\

\noindent
Let $\ldots, \phi ,~\psi , ~\phi , ~\psi ,\ldots$, $\phi  \neq \psi $ be a prime period two solution of the equation (6). We set $$u_n=w_{n-1}$$ $$v_n=w_{n}$$
\noindent
Then the equivalent form of the difference equation (6) is
$$u_{n+1}=v_n$$ $$v_{n+1}=\frac{1}{1+ p v_{n} + q u_n}$$
\noindent
Let T be the map on $\mathbb{C}\times\mathbb{C}$ to itself defined by $$T\left(
                                                                           \begin{array}{c}
                                                                             u \\
                                                                             v \\
                                                                           \end{array}
                                                                         \right)=\left(
                                                                                    \begin{array}{c}
                                                                                      v \\
                                                                                      \frac{1}{1+ pv+qu} \\
                                                                                    \end{array}
                                                                                  \right)
                                                                         $$
\noindent
Then $\left(
                                                                           \begin{array}{c}
                                                                             \phi \\
                                                                             \psi \\
                                                                           \end{array}
                                                                         \right)$ is a fixed point of $T^2$, the second iterate of $T$. \\

$$T^2\left(
                                                                           \begin{array}{c}
                                                                             u \\
                                                                             v \\
                                                                           \end{array}
                                                                         \right)=\left(
                                                                                    \begin{array}{c}
                                                                                      \frac{1}{1+ pv+qu} \\\\
                                                                                      \frac{1}{1+p \frac{1}{1+ pv+qu}+qv} \\
                                                                                    \end{array}
                                                                                  \right)
                                                                         $$

$$T^2\left(
                                                                           \begin{array}{c}
                                                                             u \\
                                                                             v \\
                                                                           \end{array}
                                                                         \right)==\left(
                                                                                    \begin{array}{c}
                                                                                      g(u,v) \\
                                                                                      h(u,v) \\
                                                                                    \end{array}
                                                                                  \right)
                                                                         $$

\noindent
 where $g(u,v)=\frac{1}{1+ pv+qu}$ and $h(u,v)=\frac{1}{1+\frac{p}{1+\text{pv}+\text{qu}}+\text{qv}}$. Clearly the two cycle is locally asymptotically stable when the eigenvalues of the
Jacobian matrix $J_{T^2}$, evaluated at $\left(
                                                                           \begin{array}{c}
                                                                             \phi \\
                                                                             \psi \\
                                                                           \end{array}
                                                                         \right)$ lie inside the unit disk.\\
                                                                         
\noindent                                                                         
We have, $$J_{T^2}\left(
                                                                           \begin{array}{c}
                                                                             \phi \\
                                                                             \psi \\
                                                                           \end{array}
                                                                         \right)=\left(
                                                                                   \begin{array}{cc}
                                                                                     \frac{\delta g}{\delta u }(\phi, \psi) & \frac{\delta g}{\delta v}(\phi, \psi) \\\\
                                                                                     \frac{\delta h}{\delta u }(\phi, \psi) & \frac{\delta h}{\delta v }(\phi, \psi) \\
                                                                                   \end{array}
                                                                                 \right)
                                                                         $$

\noindent
where $\frac{\delta g}{\delta u }(\phi, \psi)=-\frac{q}{(1+q \phi +p \psi )^2}$ and $\frac{\delta g}{\delta v }(\phi, \psi)=-\frac{p}{(1+q \phi +p \psi )^2}$\\

$\frac{\delta h}{\delta u }(\phi, \psi)=\frac{p q}{\left((1+q \phi ) (1+q \psi )+p \left(1+\psi +q \psi ^2\right)\right)^2}$
and $\frac{\delta h}{\delta v }(\phi, \psi)=\frac{-q+\frac{p^2}{(1+q \phi +p \psi )^2}}{\left(1+q \psi +\frac{p}{1+q \phi +p \psi }\right)^2}$
\\ \\
\noindent
Now, set $$\chi=\frac{\delta g}{\delta u }(\phi, \psi)+\frac{\delta h}{\delta v }(\phi, \psi)=-\frac{q}{(1+q \phi +p \psi )^2}+\frac{-q+\frac{p^2}{(1+q \phi +p \psi )^2}}{\left(1+q \psi +\frac{p}{1+q \phi +p \psi }\right)^2}$$ $$\lambda=\frac{\delta g}{\delta u }(\phi, \psi) \frac{\delta h}{\delta v }(\phi, \psi)-\frac{\delta g}{\delta v }(\phi, \psi) \frac{\delta h}{\delta u }(\phi, \psi)=\frac{q^2}{\left((1+q \phi ) (1+q \psi )+p \left(1+\psi +q \psi ^2\right)\right)^2}$$

\bigskip

\bigskip

\noindent
In particular for the prime period $2$ solution, \\ \\ $\left\{\phi \to \frac{0.5 \left(p-q+\sqrt{p^2+p (-2 -4 q) q+q^2+4 q^3}\right)}{q (-p+q)}, \psi \to \frac{0.5 p-0.5 q-0.5 \sqrt{p^2+p (-2-4 q) q+q^2+4 q^3}}{q (-p+q)}\right\}$, we shall see the local asymptotic stability for some example cases of parameters $p$ and $q$. The general form of $\chi$ and $\lambda$ would be very complected.
Consider the prime period two solution of the difference equation (6), $\phi \to -0.0843748+0.0622145 i,\psi \to -0.0822456-0.0594374 i$ corresponding two the parameters $p\to 100+i$ and $q\to 6+0.1i$. \\\\
\noindent
In this case, $\abs{\chi}=0.0735211$ and $\abs{\lambda}=0.004075$. Therefore, by the Linear Stability theorem ($|\chi| < 1+ |\lambda| < 2$) the prime period $2$ solution $\phi \to -0.0843748+0.0622145 i,\psi \to -0.0822456-0.0594374 i$ is \emph{locally asymptotically stable}.

\subsection{Prime Period Two Solutions of Eq. (7)}

Let $\ldots, \phi ,~\psi , ~\phi , ~\psi ,\ldots$, $\phi  \neq \psi $ be a prime period two solution of the difference equation $w_{n+1}=\frac{w_n}{1 + p w_n + q w_{n-1}}$. Then $\phi=\frac{\psi}{1+p \psi+q \phi}$ and $\psi=\frac{\phi}{1+p \phi +q \psi}$. This two equations lead to the set of solutions (prime period two) except the equilibriums as $\left\{\phi \to \frac{1}{0.5 p-0.5 q-0.5 \sqrt{p^2-q^2}},\psi \to -\frac{1}{q}+\frac{\sqrt{p^2-q^2}}{p q- q^2}\right\},\left\{\phi \to \frac{1}{0.5 p-0.5 q+0.5 \sqrt{p^2-q^2}}, \psi \to -\frac{1}{q}-\frac{\sqrt{p^2-q^2}}{p q-q^2}\right\}$.\\

\noindent
Let $\ldots, \phi ,~\psi , ~\phi , ~\psi ,\ldots$, $\phi  \neq \psi $ be a prime period two solution of the equation (7). We set $$u_n=w_{n-1}$$ $$v_n=w_{n}$$
\noindent
Then the equivalent form of the difference equation (7) is
$$u_{n+1}=v_n$$ $$v_{n+1}=\frac{v_n}{1+ p v_{n} + q u_n}$$
\noindent
Let T be the map on $\mathbb{C}\times\mathbb{C}$ to itself defined by $$T\left(
                                                                           \begin{array}{c}
                                                                             u \\
                                                                             v \\
                                                                           \end{array}
                                                                         \right)=\left(
                                                                                    \begin{array}{c}
                                                                                      v \\
                                                                                      \frac{v}{1+ pv+qu} \\
                                                                                    \end{array}
                                                                                  \right)
                                                                         $$
\noindent
Then $\left(
                                                                           \begin{array}{c}
                                                                             \phi \\
                                                                             \psi \\
                                                                           \end{array}
                                                                         \right)$ is a fixed point of $T^2$, the second iterate of $T$. \\

$$T^2\left(
                                                                           \begin{array}{c}
                                                                             u \\
                                                                             v \\
                                                                           \end{array}
                                                                         \right)=\left(
                                                                                    \begin{array}{c}
                                                                                      \frac{v}{1+ pv+qu} \\\\
                                                                                      \frac{\frac{v}{1+ pv+qu}}{1+p \frac{v}{1+ pv+qu} +qv} \\
                                                                                    \end{array}
                                                                                  \right)
                                                                         $$

$$T^2\left(
                                                                           \begin{array}{c}
                                                                             u \\
                                                                             v \\
                                                                           \end{array}
                                                                         \right)==\left(
                                                                                    \begin{array}{c}
                                                                                      g(u,v) \\
                                                                                      h(u,v) \\
                                                                                    \end{array}
                                                                                  \right)
                                                                         $$

\noindent
 where $g(u,v)=\frac{v}{1+ pv+qu}$ and $h(u,v)=\frac{\frac{v}{1+ pv+qu}}{1+p \frac{v}{1+ pv+qu} +qv}$. Clearly the two cycle is locally asymptotically stable when the eigenvalues of the
Jacobian matrix $J_{T^2}$, evaluated at $\left(
                                                                           \begin{array}{c}
                                                                             \phi \\
                                                                             \psi \\
                                                                           \end{array}
                                                                         \right)$ lie inside the unit disk.\\
                                                                         
\noindent                                                                          
We have, $$J_{T^2}\left(
                                                                           \begin{array}{c}
                                                                             \phi \\
                                                                             \psi \\
                                                                           \end{array}
                                                                         \right)=\left(
                                                                                   \begin{array}{cc}
                                                                                     \frac{\delta g}{\delta u }(\phi, \psi) & \frac{\delta g}{\delta v}(\phi, \psi) \\\\
                                                                                     \frac{\delta h}{\delta u }(\phi, \psi) & \frac{\delta h}{\delta v }(\phi, \psi) \\
                                                                                   \end{array}
                                                                                 \right)
                                                                         $$

\noindent
where $\frac{\delta g}{\delta u }(\phi, \psi)=-\frac{q \psi }{(1+q \phi +p \psi )^2}$ and $\frac{\delta g}{\delta v }(\phi, \psi)=\frac{1+q \phi }{(1+q \phi +p \psi )^2}$\\

$\frac{\delta h}{\delta u }(\phi, \psi)=-\frac{q \psi  (1+q \psi )}{\left(1+\psi +p \psi +q \left(\phi +\psi +q \phi  \psi +p \psi ^2\right)\right)^2}$
and $\frac{\delta h}{\delta v }(\phi, \psi)=\frac{1+q \left(\phi -p \psi ^2\right)}{\left(1+\psi +p \psi +q \left(\phi +\psi +q \phi  \psi +p \psi ^2\right)\right)^2}$
\\ \\
\noindent
Now, set $$\chi=\frac{\delta g}{\delta u }(\phi, \psi)+\frac{\delta h}{\delta v }(\phi, \psi)=-\frac{q \psi }{(1+q \phi +p \psi )^2}+\frac{1+q \left(\phi -p \psi ^2\right)}{\left(1+\psi +p \psi +q \left(\phi +\psi +q \phi  \psi +p \psi ^2\right)\right)^2}$$ $$\lambda=\frac{q^2 \psi ^2}{(1+q \phi +p \psi ) \left(1+\psi +p \psi +q \left(\phi +\psi +q \phi  \psi +p \psi ^2\right)\right)^2}$$

\bigskip

\bigskip

\noindent
In particular for the prime period $2$ solution, \\ \\ $\left\{\phi \to \frac{1}{0.5 p-0.5 q-0.5 \sqrt{p^2-q^2}},\psi \to -\frac{1}{q}+\frac{\sqrt{p^2-q^2}}{p q- q^2}\right\}$, we shall see the local asymptotic stability for some example cases of parameters $p$ and $q$. The general form of $\chi$ and $\lambda$ would very complected.
Consider the prime period two solution of the difference equation (7), $\phi \to 0.365026 +0.263198i ,\psi \to -0.412345 + 0.131124i$ corresponding two the parameters $p \to \frac{1}{5} + 3i$ and $q \to \frac{3}{5}+5i$. \\\\
\noindent
In this case, $\abs{\chi}=0.0287948$ and $\abs{\lambda}=0.0000431717$. Therefore, by the Linear Stability theorem ($|\chi| < 1+ |\lambda| < 2$) the prime period $2$ solution $\phi \to 0.365026 +0.263198i ,\psi \to -0.412345 + 0.131124i$ is \emph{locally asymptotically stable}.

\subsection{Prime Period Two Solutions of Eq. (8)}

Let $\ldots, \phi ,~\psi , ~\phi , ~\psi ,\ldots$, $\phi  \neq \psi $ be a prime period two solution of the difference equation $w_{n+1}=\frac{w_{n-1}}{p + q w_n + w_{n-1}}$. Then $\phi=\frac{\phi}{1+p \psi+q \phi}$ and $\psi=\frac{\psi}{1+p \phi +q \psi}$. This two equations lead to the set of solutions (prime period two) except the equilibriums as $\left\{\phi \to 0, \psi \to 1-p\right\}$.\\

\noindent
Let $\ldots, \phi ,~\psi , ~\phi , ~\psi ,\ldots$, $\phi  \neq \psi $ be a prime period two solution of the equation (8). We set $$u_n=w_{n-1}$$ $$v_n=w_{n}$$
\noindent
Then the equivalent form of the difference equation (8) is
$$u_{n+1}=v_n$$ $$v_{n+1}=\frac{u_n}{p+ q v_{n} +  u_n}$$
\noindent
Let T be the map on $\mathbb{C}\times\mathbb{C}$ to itself defined by $$T\left(
                                                                           \begin{array}{c}
                                                                             u \\
                                                                             v \\
                                                                           \end{array}
                                                                         \right)=\left(
                                                                                    \begin{array}{c}
                                                                                      v \\
                                                                                      \frac{u}{p+ qv+u} \\
                                                                                    \end{array}
                                                                                  \right)
                                                                         $$
\noindent
Then $\left(
                                                                           \begin{array}{c}
                                                                             \phi \\
                                                                             \psi \\
                                                                           \end{array}
                                                                         \right)$ is a fixed point of $T^2$, the second iterate of $T$. \\

$$T^2\left(
                                                                           \begin{array}{c}
                                                                             u \\
                                                                             v \\
                                                                           \end{array}
                                                                         \right)=\left(
                                                                                    \begin{array}{c}
                                                                                      \frac{u}{p+ qv+u} \\\\
                                                                                      \frac{\frac{u}{p+ qv+u}}{p+q \frac{u}{p+ qv+u}+v} \\
                                                                                    \end{array}
                                                                                  \right)
                                                                         $$

$$T^2\left(
                                                                           \begin{array}{c}
                                                                             u \\
                                                                             v \\
                                                                           \end{array}
                                                                         \right)==\left(
                                                                                    \begin{array}{c}
                                                                                      g(u,v) \\
                                                                                      h(u,v) \\
                                                                                    \end{array}
                                                                                  \right)
                                                                         $$

\noindent
 where $g(u,v)=\frac{u}{p+ qv+u}$ and $h(u,v)=\frac{\frac{u}{p+ qv+u}}{p+q \frac{u}{p+ qv+u}+v}$. Clearly the two cycle is locally asymptotically stable when the eigenvalues of the
Jacobian matrix $J_{T^2}$, evaluated at $\left(
                                                                           \begin{array}{c}
                                                                             \phi \\
                                                                             \psi \\
                                                                           \end{array}
                                                                         \right)$ lie inside the unit disk.\\
                                                                         
\noindent                                                                         
We have, $$J_{T^2}\left(
                                                                           \begin{array}{c}
                                                                             \phi \\
                                                                             \psi \\
                                                                           \end{array}
                                                                         \right)=\left(
                                                                                   \begin{array}{cc}
                                                                                     \frac{\delta g}{\delta u }(\phi, \psi) & \frac{\delta g}{\delta v}(\phi, \psi) \\\\
                                                                                     \frac{\delta h}{\delta u }(\phi, \psi) & \frac{\delta h}{\delta v }(\phi, \psi) \\
                                                                                   \end{array}
                                                                                 \right)
                                                                         $$

\noindent
where $\frac{\delta g}{\delta u }(\phi, \psi)=\frac{p+q \psi }{(p+\phi +q \psi )^2}$ and $\frac{\delta g}{\delta v }(\phi, \psi)=-\frac{q \phi }{(p+\phi +q \psi )^2}$\\

$\frac{\delta h}{\delta u }(\phi, \psi)=\frac{(1+\psi ) (p+q \psi )}{\left(p+\phi +q \phi +(p+q+\phi ) \psi +q \psi ^2\right)^2}$
and $\frac{\delta h}{\delta v }(\phi, \psi)=-\frac{\phi  (p+q+\phi +2 q \psi )}{\left(p+\phi +q \phi +(p+q+\phi ) \psi +q \psi ^2\right)^2}$
\\ \\
\noindent
Now, set $$\chi=\frac{\delta g}{\delta u }(\phi, \psi)+\frac{\delta h}{\delta v }(\phi, \psi)=\frac{p+q \psi }{(p+\phi +q \psi )^2}-\frac{\phi  (p+q+\phi +2 q \psi )}{\left(p+\phi +q \phi +(p+q+\phi ) \psi +q \psi ^2\right)^2}$$ $$\lambda=-\frac{\phi  (p+q \psi )}{(p+\phi +q \psi ) \left(p+\phi +q \phi +(p+q+\phi ) \psi +q \psi ^2\right)^2}$$

\bigskip

\bigskip

\noindent
For the prime period $2$ solution, $\phi \to 0, \psi \to 1-p$, $\chi=\frac{1}{p+(1-p) q}$ and $\abs{\lambda}=0$. Therefore, by the Linear Stability theorem ($|\chi| < 1+ |\lambda| < 2$) the prime period $2$ solution is \emph{locally asymptotically stable} if and only if $\abs{\frac{1}{p+(1-p) q}}<1$. It turns out that $\abs{q}>\frac{1-\abs{p}}{\abs{1-p}}$. In other words, the condition reduces to $\abs{p}<1$ and $\abs{q}>1$ which is same condition as it was for the real set up.

\section{Chaotic Solutions}
This is something which is absolutely new feature of the dynamics of the difference equation (1) which did not arise in the real set up of the same difference equation. Computationally we have encountered some chaotic solution of the difference equation (8) for some parameter values which are given in the following Table. 1. \\
\noindent
The method of Lyapunov characteristic exponents serves as a useful tool to quantify chaos. Specifically Lyapunav exponents measure the rates of convergence or divergence of nearby trajectories. Negative Lyapunov exponents indicate convergence, while positive Lyapunov exponents demonstrate divergence and chaos. The magnitude of the Lyapunov exponent is an indicator of the time scale on which chaotic behavior can be predicted or transients decay for the positive and negative exponent cases respectively. In this present study, the largest Lyapunov exponent is calculated for a given solution of finite length numerically \cite{Wolf}.\\
\noindent
From computational evidence, it is arguable that for complex parameters $p$ and $q$ which are stated in the following table the solutions are chaotic for every initial values.

\begin{table}[H]

\begin{tabular}{| m{7cm} | m{8cm} |}
\hline
\centering   \textbf{Parameters} $p$, $q$ &
\begin{center}
\textbf{Interval of Lyapunav exponent}
\end{center}\\
\hline
\centering $p=(0.2037, 0.5444)$, $q=(0.8749, 0.1210)$&
\begin{center}
$(0.3215, 1.6235)$
\end{center}\\
\hline
\centering $p=(0.4933, 0.7018)$, $q=(0.8878, 0.0551)$ &
\begin{center}
$(1.062, 2.021)$
\end{center}\\
\hline
\centering $p=(0.7840, 0.4867)$, $q=(0.4648, 0.1313)$ &
\begin{center}
$(0.6256, 1.314)$
\end{center} \\
\hline
\centering $p=(0.2308, 0.6580)$, $q=(0.5629, 0.2818)$ &
\begin{center}
$(1.373, 2.325)$
\end{center}\\
\hline

\end{tabular}
\caption{Chaotic solutions of the equation (8) for different choice of parameters and initial values.}
\label{Table:}
\end{table}

\noindent
The chaotic trajectory plots including corresponding complex plots are given the following Fig. 1.

\begin{figure}[H]
      \centering

      \resizebox{16cm}{!}
      {
      \begin{tabular}{c}
      \includegraphics [scale=5]{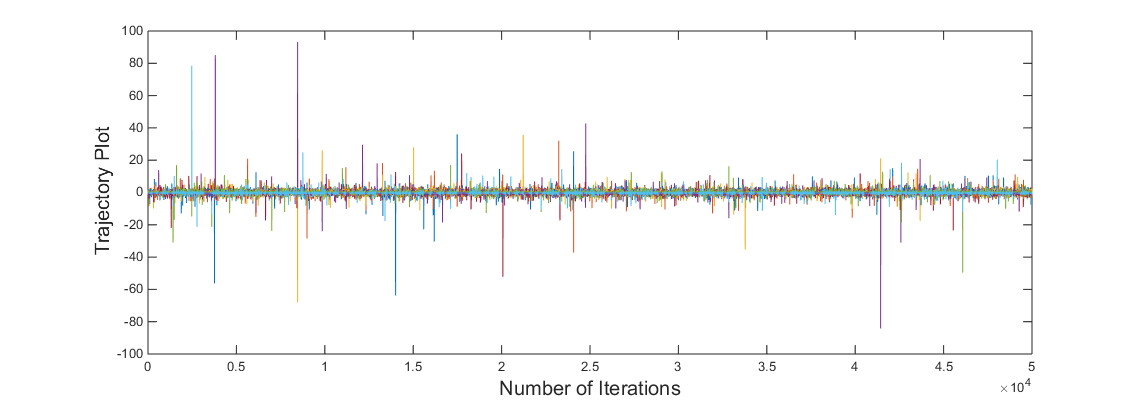}
      \includegraphics [scale=4]{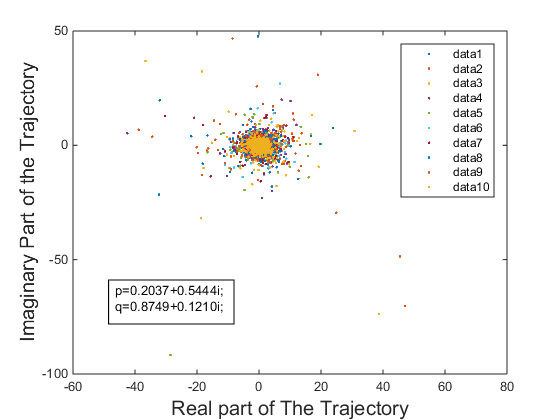}\\
      \includegraphics [scale=5]{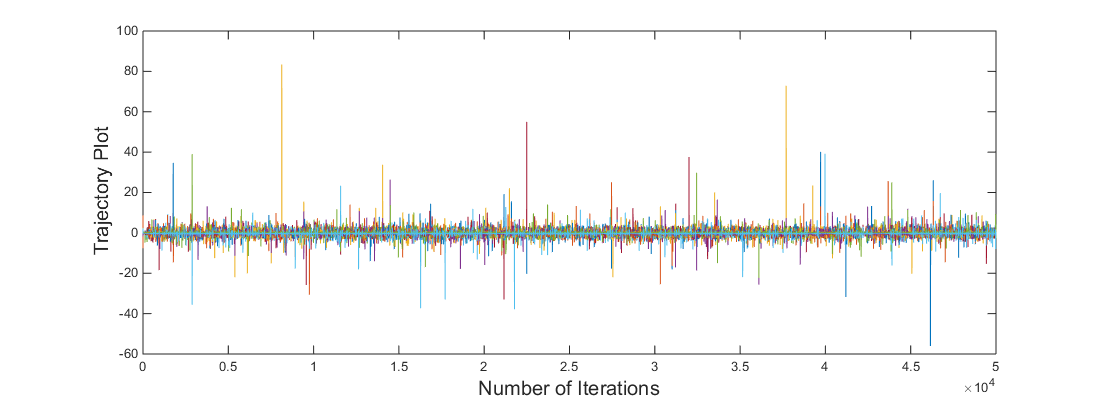}
      \includegraphics [scale=4]{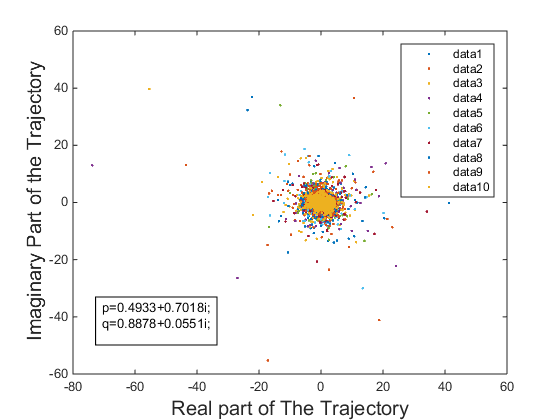}\\
      \includegraphics [scale=5]{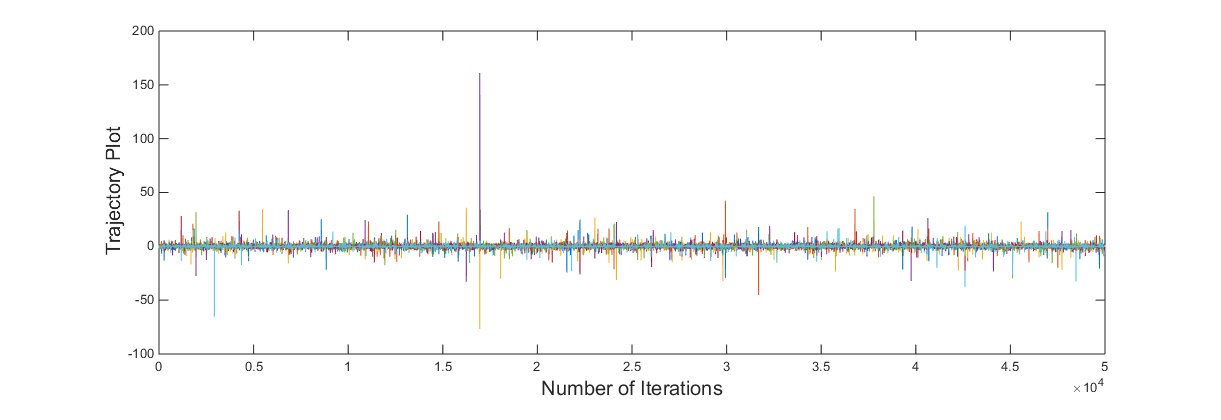}
      \includegraphics [scale=4]{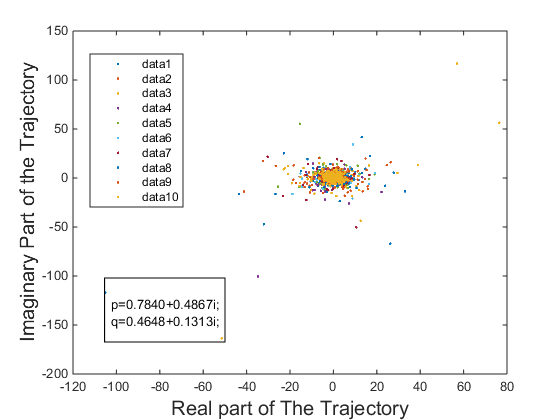}\\
      \includegraphics [scale=5]{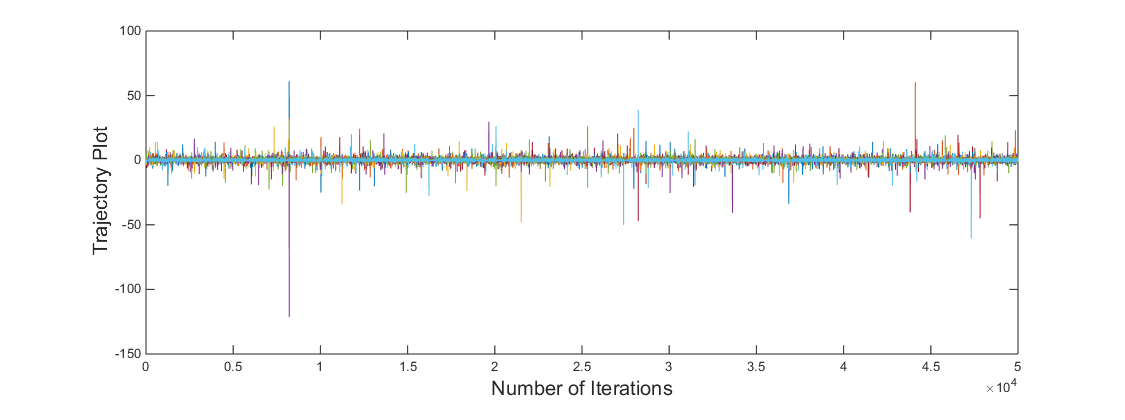}
      \includegraphics [scale=4]{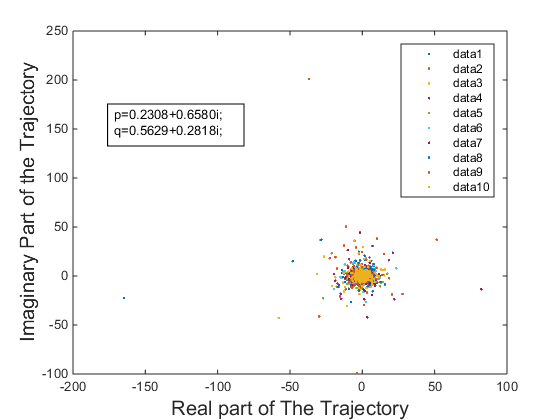}\\
      \end{tabular}
      }
\caption{Chaotic Trajectories of the equation (8) of four different cases as stated in Table 1.}
      \begin{center}

      \end{center}
      \end{figure}

In the Fig. 1, for each of the four cases ten different initial values are taken and plotted in the left and in the right corresponding complex plots are given. From the Fig. 1, it is evident that for the four different cases the basin of the chaotic attractor is neighbourhood of the centre $(0, 0)$ of complex plane.

\section{Some Interesting Nontrivial Problems}

\begin{openproblem}
Does the difference equation have higher order periodic cycle? If so, what is the highest periodic cycle?
\end{openproblem}

\begin{openproblem}
Find out the set of all parameters $p$ and $q$ for which the difference equation (8) has chaotic solutions.
\end{openproblem}

\begin{openproblem}
Find out the subset of the $\mathbb{D}$ of all possible initial values $z_0$ and $z_1$ for which the solutions of the difference equation are chaotic for any complex parameters $p$ and $q$. Does the neighbourhood of $(0, 0)$ is global chaotic attractor? If not, are there any other chaotic attractors?
\end{openproblem}

\section{Future Endeavours}

In continuation of the present work the study of the difference equation ${z_{n+1}=\frac{\alpha_n + \beta_n z_{n}+ \gamma_n z_{n-1}}{A_n + B_n z_n + C_n z_{n-1}}}$ where $\alpha_n$, $\beta_n$, $\gamma_n$, $A_n$, $B_n$ and $C_n$ are all convergent sequence of complex numbers and converges to $\alpha$, $\beta$, $\gamma$, $A$, $B$ and $C$ respectively is indeed would be very interesting and that we would like to pursue further. Also the most generalization of the present rational difference equation is $${z_{n+1}=\frac{\alpha + \beta z_{n-l}+ \gamma z_{n-k}}{A + B z_{n-l} + C z_{n-k}}}$$ where $l$ and $k$ are delay terms and it demands similar analysis which we plan to pursue in near future.

\section*{Acknowledgement}
The author thanks \emph{Dr. Pallab Basu} for discussions and suggestions.


\end{document}